\documentclass[a4paper,fleqn]{cas-dc}

\usepackage[authoryear]{natbib}

\ExplSyntaxOn

\ExplSyntaxOff

\usepackage{cleveref}
\usepackage{graphicx}
\usepackage{booktabs}
\usepackage{multirow}
\usepackage{array}
\usepackage{stfloats} 
\usepackage{cuted} 
\usepackage{algorithm}
\usepackage{algpseudocode}
\usepackage{amsmath}
\usepackage{capt-of} 

\usepackage{amsthm}
\theoremstyle{plain}
\newtheorem{theorem}{Theorem}
\newtheorem{lemma}{Lemma}

\newtheorem{proposition}{Proposition}
\theoremstyle{definition}
\newtheorem{remark}{Remark}
\newtheorem{example}{Example}

\makeatletter
\renewenvironment{proof}[1][\proofname]{\par
	\pushQED{\qed}%
	\normalfont
	\topsep6\p@\@plus6\p@\relax
	\trivlist
	\item[\hskip\labelsep\bfseries #1\@addpunct{.}]%
}{%
	\popQED\endtrivlist\@endpefalse
}
\makeatother

\usepackage{caption} 
\def\tsc#1{\csdef{#1}{\textsc{\lowercase{#1}}\xspace}}
\tsc{WGM}
\tsc{QE}

\begin{document}
\let\WriteBookmarks\relax
\setcounter{topnumber}{5}
\setcounter{bottomnumber}{5}
\setcounter{totalnumber}{10}
\renewcommand{\topfraction}{0.95}
\renewcommand{\bottomfraction}{0.85}
\renewcommand{\textfraction}{0.05}
\renewcommand{\floatpagefraction}{0.80}
\renewcommand{\dbltopfraction}{0.95}
\renewcommand{\dblfloatpagefraction}{0.80}

\shorttitle{An Inner-Outer Iteration Algorithm for Stochastic Lyapunov Matrix Equation}    

\shortauthors{He et~al.}  

\title [mode = title]{An Inner-Outer Iteration Algorithm with Optimal Parameters for Stochastic Lyapunov Matrix Equation}  

\tnotemark[1]
\tnotetext[1]{This work was supported by the Guangdong Basic and Applied Basic Research Foundation of China (No. 2026A1515012144).}
\author[]{Donghuan He}
\credit{Conceptualization, Methodology, Writing -- original draft}

\author[]{Xiaowen Su}
\credit{Methodology, Validation}

\author[]{Feng Wang}
\credit{Software, Formal analysis}

\author[]{Xuesong Chen}
\cormark[1]
\ead{chenxs@gdut.edu.cn}
\credit{Supervision, Writing -- review and editing}

\affiliation[]{organization={School of Mathematics and Statistics, Guangdong University of Technology},
	city={Guangzhou},
	postcode={510520},
	country={P.R. China}}

\cortext[1]{Corresponding author}

\begin{abstract}
This paper proposes an inner--outer (IO) iterative algorithm with optimal parameters for solving stochastic Lyapunov matrix equation associated with discrete-time stochastic linear system. First, under the assumption that the underlying stochastic linear system is asymptotically mean-square stable, the monotonicity and boundedness of the  iterative sequence generated by the proposed algorithm are analyzed. On this basis, a sufficient convergence result is established for the zero initial condition. Second, by deriving the spectral radius of the corresponding iteration matrix, several necessary and sufficient convergence conditions are obtained for arbitrary initial conditions. In addition, the optimal parameter-selection strategies are developed to improve the convergence performance of the algorithm. Finally, numerical examples are presented to verify the theoretical results and demonstrate the advantages of the proposed algorithm over several existing iterative methods.
\end{abstract}


\maketitle

\section{Introduction}
 Stochastic linear systems have been widely applied in various fields.For instance, they have been used to construct dynamic segment models for speech recognition \cite{digalakis2002ml} and applied to real-time control with deadline overruns \cite{gallant2025soft}. Owing to their importance, the stability, observability, and detectability of stochastic linear systems have attracted considerable attention. \cite{mclane1969asymptotic} introduced the concept of asymptotic mean-square stability (AMSS) for linear stochastic systems. It was further proved that the existence of a positive definite solution to the corresponding Lyapunov matrix (LM) equation provides a necessary and sufficient condition for the AMSS of the system. These results indicate that the properties of stochastic linear systems are closely related to those of the solutions to their corresponding LM equations. Such equations are referred to as stochastic Lyapunov matrix (SLM) equations. Therefore, developing efficient methods for solving SLM equations is of great importance.

A direct method for solving LM equations is to transform the matrix equations into systems of linear equations using the Kronecker product \cite{jodar1987explicit}. In this way, various methods for solving linear systems can be applied to the resulting matrix equation. However, because this approach involves Kronecker product operations, its computational cost increases significantly for large-scale matrices. Therefore, iterative methods provide an effective alternative for solving LM equations. Several gradient-based iterative algorithms  (\cite{ding2005gradient,zhou2010gradient}) have been developed for solving LM equations. An implicit iterative algorithm for solving SLM equation was proposed in \cite{zhang2017implicit}. An iterative algorithm for discrete periodic LM equation was proposed in \cite{wu2018iterative}. For Markovian jump LM equations,  \cite{tian2018new} proposed an inner--outer (IO) iterative algorithm for solving the CCMJLM equation. By applying multi-step Smith iterations to the inner iteration, \cite{tian2020multi} developed a multi-step Smith IO algorithm for solving the CCMJLM equation. \cite{he2025convergence} proposed a Jacobi gradient-based iterative algorithm for solving the complex conjugate and transpose Sylvester matrix equation. This method can also be applied to LM equation.

Recently, \cite{chen2026explicit} proposed an explicit iterative algorithm with an optimal tuning parameter for solving SLM equation. They compared this method with the implicit algorithm proposed in \cite{zhang2017implicit}. However, both the explicit and implicit iterative algorithms involve only one tuning parameter, which may limit further improvement in convergence performance. Therefore, it is desirable to develop an algorithm with an additional adjustable parameter. Inspired by the IO algorithm proposed in \cite{tian2018new}, this paper develops an IO algorithm for solving SLM equation. Since the proposed algorithm includes an additional parameter, namely the number of inner iteration steps, its parameters can be tuned more flexibly to improve convergence performance.

\textbf{Notation}: $A^T$, $\rho(A)$, and $\sigma(A)$ denote the transpose, spectral radius, and spectrum of $A$, respectively. 
$A\otimes B$ denotes the Kronecker product. 
$A\succeq0$ $(\succ0)$ means that $A$ is positive semidefinite (positive definite), and $A\geq0$ means that all entries of $A$ are nonnegative. 
The notation $i\in\{1,2,\dots,n\}$ is abbreviated as $i\in\mathbb{I}[1,\,n]$. The operators $\operatorname{vec}(\cdot)$ and $\mathbf{E}[\cdot]$ denote vectorization and mathematical expectation, respectively.

\section{Previous results and the proposed algorithm}\label{previous}
Consider the following discrete-time stochastic linear system:
\begin{equation}\label{sto_lin_sys}
	x(t+1) = A_0 x(t) + \sum_{i=1}^{m} A_i x(t) \omega_i(t),
\end{equation}
where $x(t) \in \mathbb{R}^n$ is the state vector, $A_0, A_1, \dots, A_m \in \mathbb{R}^{n \times n}$ are real constant matrices, $\omega_i(t) \in \mathbb{R}, i\in\mathbb{I}[1,\,m]$ are independent random disturbances satisfying $\mathbf{E}[\omega_i(t)] = 0$, $\mathbf{E}[\omega_i^2(t)] = \delta_i$, $\mathbf{E}[\omega_i(t)\omega_j(t)] = 0$, $i \neq j$, $i, j \in \mathbb{I}[1,\,m]$. The stability of the system (\ref{sto_lin_sys}) can be characterized by the following SLM equation,
\begin{equation}\label{sto_lya_eq}
	A_0^\mathrm{T} X A_0 + \sum_{i=1}^{m} \delta_i A_i^\mathrm{T} X A_i - X = -Q, 
\end{equation}
where $A_0, A_1, \dots, A_m \in \mathbb{R}^{n \times n}$ and $Q \succ 0$ are the coefficient matrices and $X \in \mathbb{R}^{n \times n}$ is the unknown matrix.
\begin{lemma}\label{stab_condi}
	\cite{kubrusly1985mean} Consider the discrete-time stochastic system (\ref{sto_lin_sys}). The following assertions are equivalent.\\
		1) The discrete-time stochastic system (\ref{sto_lin_sys}) is asymptotically mean-square stable. \\
		2) For any given positive definite matrix $Q$, the SLM equation (\ref{sto_lya_eq}) with respect to $X$ has a unique positive definite solution. \\
		3) $\rho(\Phi) < 1$, where the matrix $\Phi$ is defined as
		\begin{equation}
			\Phi = A_0^\mathrm{T} \otimes A_0^\mathrm{T} + \sum_{i=1}^{m} \delta_i A_i^\mathrm{T} \otimes A_i^\mathrm{T}.
		\end{equation}
\end{lemma}

To solve the SLM equation \eqref{sto_lya_eq}, a direct approach is to employ the Smith iteration, which is given as follows:

\begin{equation}\label{simith}
	X_{k+1} = A_0^\mathrm{T} X_k A_0 + \sum_{i=1}^{m} \delta_i A_i^\mathrm{T} X_k A_i + Q.
\end{equation}
However, this algorithm contains no tuning parameter. Thus, its convergence performance cannot be improved through parameter adjustment.

In \cite{zhang2017implicit}, an implicit iterative method was proposed for solving the SLM equation with a single disturbance term. This method can be naturally extended to the SLM equation \eqref{sto_lya_eq}. The corresponding iterative formula is given as follows:
\begin{equation}\label{zhan}
	\begin{aligned}
		A_0^\mathrm{T} X_{k+1} A_0 - X_{k+1}
		= {} & -\gamma 
		\left( 
		\sum_{i=1}^{m} \delta_i A_i^\mathrm{T} X_{k} A_i 
		+ Q 
		\right)  \\
		& + (1-\gamma) 
		\left( 
		A_0^\mathrm{T} X_{k} A_0 - X_{k} 
		\right).
	\end{aligned}
\end{equation}

Since algorithm \eqref{zhan} is implicit, a standard discrete LM equation must be solved at each iteration. Therefore, Chen et al. \cite{chen2026explicit} constructed an explicit iterative formula for solving matrix equation \eqref{sto_lya_eq}. The algorithm is given as follows:
\begin{equation}\label{chen}
	\begin{aligned}
		X_{k+1} = & \gamma \left( A_0^\mathrm{T} X_{k} A_0 + \sum_{i=1}^{m} \delta_i A_i^\mathrm{T} X_{k} A_i + Q \right) \\ & + (1-\gamma) X_{k}.
	\end{aligned}
\end{equation}

However, both Algorithm \ref{zhan} and Algorithm \ref{chen} contain only one tuning parameter, $\gamma$, which limits their flexibility in improving convergence performance. Therefore, an algorithm with an additional adjustable parameter is developed below.

\cite{tian2018new} proposed an inner-outer (IO) iterative algorithm for solving the CCMJLM equation. Inspired by the construction of the IO algorithm, we apply this idea to solve the SLM equation \eqref{sto_lya_eq}. Applying the vectorization operator to \eqref{sto_lya_eq} yields
\begin{equation}\label{sto_lya_eq_kro}
	(I-\Phi)\operatorname{vec}(X)=\operatorname{vec}(Q).
\end{equation}
For a given parameter $\alpha\in \mathbb{R}$, we have the following identity:
\begin{equation}\label{eq:matrix_splitting}
	I-\Phi=(I-\alpha\Phi)-(1-\alpha)\Phi.
\end{equation}
Based on \eqref{eq:matrix_splitting}, the outer iteration for solving \eqref{sto_lya_eq_kro} is constructed as follows:
\begin{equation}\label{eq:outer_vec}
	\begin{aligned}
	(I-\alpha\Phi)\operatorname{vec}(X_{k+1})
	=
	&(1-\alpha)\Phi\operatorname{vec}(X_k)
	+
	\operatorname{vec}(Q).
	\end{aligned}
\end{equation}
Applying the inverse vectorization operation, \eqref{eq:outer_vec} is equivalent to
\begin{equation}\label{eq:outer_matrix}
	X_{k+1}
	-
	\alpha
	\left(
	A_0^{T}X_{k+1}A_0
	+
	\sum_{i=1}^{m}\delta_i A_i^{T}X_{k+1}A_i
	\right)
	=
	W_k,
\end{equation}
where
\[
W_k
=
(1-\alpha)
\left(
A_0^{T}X_kA_0
+
\sum_{i=1}^{m}\delta_i A_i^{T}X_kA_i
\right)
+
Q.
\]
To obtain the iterative estimate $X_{k+1}$, the following inner iteration is employed:
\begin{equation}\label{sto_inner_eq}
	\begin{aligned}
			Y_{j+1}
		=
		& \alpha
		\left(
		A_0^{T}Y_jA_0
		+
		\sum_{i=1}^{m}\delta_i A_i^{T}Y_jA_i
		\right)
		+
		W_k, \\ 
		&j\in \mathbb{I}[0,\,l-1].
	\end{aligned}
\end{equation}
where the initial condition is $Y_0=X_k$, and $Y_l$ is taken as an approximation of $X_{k+1}$.

In the inner iteration \eqref{sto_inner_eq}, replacing $Y_j$ with $X_{k,j}$, where $X_{k,j}$ denotes the iterative solution generated by the $j$th inner iteration within the $k$th outer iteration, yields the proposed IO iterative algorithm, which is summarized as follows:
\begin{equation}\label{IO_alg}
	\left\{
	\begin{aligned}
		X_{k,0} &= X_k,\quad X_{k,l}=X_{k+1},\\
		X_{k,j+1}
		&=
		\alpha
		\left(
		A_0^{T}X_{k,j}A_0
		+
		\sum_{i=1}^{m}\delta_i A_i^{T}X_{k,j}A_i
		\right)  \\
		&\quad
		+
		(1-\alpha)
		\left(
		A_0^{T}X_kA_0
		+
		\sum_{i=1}^{m}\delta_i A_i^{T}X_kA_i
		\right)\\
		&\quad + Q,\\
		&\hspace{-1.5em} j\in \mathbb{I}[0,\,l-1],\quad k=0,1,2,\ldots .
	\end{aligned}
	\right.
\end{equation}

Several lemmas are presented at the end of this section for the subsequent convergence analysis of the algorithm \eqref{IO_alg}.

\begin{lemma}\label{lemma converge of sequence}
	\cite{bibby1974axiomatisations} Let $\{X_k\}$ be a sequence of positive definite matrices satisfying the following conditions:
		1)  $X_k \preceq X_{k+1}$ for all integers $k \ge 0$. \\
		2)  There exists a positive definite matrix $X$ such that $X_k \preceq X$ for all integers $k \ge 0$.
	Then the sequence $\{X_k\}$ converges.
\end{lemma}

\begin{lemma}\label{Perro_Fro}
	\cite{horn2012matrix} Let $A\in\mathbb{R}^{n\times n}$ be a nonnegative matrix, then the spectral radius $\rho(A)$ is an eigenvalue of $A$.
\end{lemma}

\section{Convergence of the proposed algorithm}\label{Convergence of the proposed algorithm}
	
The convergence of the proposed algorithm \eqref{IO_alg} is analyzed in this section. First, the monotonicity and boundedness of the iterative sequence generated by Algorithm \eqref{IO_alg} are established.
For the convenience of the subsequent proof, a linear operator $\mathcal{L}: \mathbb{R}^{n\times n} \rightarrow \mathbb{R}^{n\times n}$ is introduced here and defined as follows:
\[
\mathcal{L}(X)
=
A_0^{T}XA_0+\sum_{i=1}^{m}\delta_i A_i^{T}XA_i.
\]
Since $\delta_i \geq 0$, the operator $\mathcal{L}$ is positive. That is, $X \succeq Y$ implies $\mathcal{L}(X) \succeq \mathcal{L}(Y)$.

\begin{lemma}\label{lemma mono}
Assume that the stochastic linear system \eqref{sto_lin_sys} is asymptotically mean-square stable. Let $Q$ be an arbitrary positive definite matrix in matrix equation \eqref{sto_lya_eq}, let $0<\alpha\leq1$, and let the number of inner iteration steps satisfy $l \ge 1$. Then, under the zero initial condition, both the iterative sequence $\{X_k\}$ and the inner iterative sequence $\{X_{k,j}\}$ generated by Algorithm \eqref{IO_alg} are monotonically nondecreasing, that is,
	\[
	\begin{aligned}
		&X_k\preceq X_{k+1},
		\quad X_{k,j} \preceq X_{k,j+1} \\
		 &j\in \mathbb{I}[0,\,l-1],\quad k=0,1,2,\ldots .
	\end{aligned}
	\]
\end{lemma}
    
Because of the limitation of length, the proof of this lemma is omitted 

\begin{lemma}\label{lemma bou}
	Assume that the stochastic linear system \eqref{sto_lin_sys} is asymptotically mean-square stable and that $X^*$ is the unique positive definite solution of matrix equation \eqref{sto_lya_eq}. Let $Q$ be an arbitrary positive definite matrix in the equation \eqref{sto_lya_eq}, let $0<\alpha\leq1$, and let $l \ge 1$. Then, under the zero initial condition, the iterative sequence ${X_k}$ generated by the proposed IO algorithm \eqref{IO_alg} is bounded above by $X^*$, that is,
	\[
	X_k\preceq X^*,
	\qquad k=0,1,2,\ldots .
	\]
\end{lemma}

The proof of this lemma is also omitted. Based on the above two lemmas, the following conclusion can be obtained.

\begin{theorem}\label{th_mo_bo}
	Assume that the stochastic linear system \eqref{sto_lin_sys} is asymptotically mean-square stable and that $X^{*}$ is the unique positive definite solution of matrix equation \eqref{sto_lya_eq}. Let $Q$ be an arbitrary positive definite matrix in the equation \eqref{sto_lya_eq}, let $0<\alpha\leq1$, and let $l \ge 1$. Then, under the zero initial condition, the iterative sequence ${X_k}$ generated by the proposed IO algorithm \eqref{IO_alg} converges to $X^*$.
\end{theorem}

\begin{proof}
	According to  \textbf{\Cref{lemma converge of sequence}} and \textbf{\Cref{lemma mono}}, the sequence $\{X_k\}$ is convergent.
	Let $\lim\limits_{k\to\infty} X_k = X_\infty$.
	 
	It can be easily proved by mathematical induction that
	\[
	\mathcal{R}(X_k)=\mathcal{L}(X_k)+Q-X_k\succeq 0 .
	\]
	Combining $X_{k,1}=\mathcal{L}(X_k)+Q$ with  $X_{k,1}\preceq X_{k,l} = X_{k+1}$, one obtains
	\[
	0\preceq \mathcal{R}(X_k)
	=
	X_{k,1}-X_k
	\preceq
	X_{k+1}-X_k .
	\]
	Since $\{X_k\}$ is convergent, it follows that
	\[
	\lim_{k\to\infty}(X_{k+1}-X_k) = 0,
	\]
	Consequently,
	\[
	\lim_{k\to\infty}\mathcal{R}(X_k)= \big[\mathcal{L}(X_k)+Q-X_k\big] = 0.
	\]
	That is,
	\[
	\mathcal{L}(X_\infty)+Q-X_\infty=0.
	\]
	This is equivalent to
	\[
	A_0^{T}X_\infty A_0
	+
	\sum_{i=1}^{m}\delta_i A_i^{T}X_\infty A_i
	-
	X_\infty
	=
	-Q .
	\]
Hence, $X_\infty$ is a solution of SLM equation \eqref{sto_lya_eq}. Since stochastic system \eqref{sto_lin_sys} is AMSS, SLM equation \eqref{sto_lya_eq} has a unique positive definite solution $X^{*}$. Therefore, $X_\infty=X^{*}$. This completes the proof.
\end{proof}

Since the convergence of Algorithm~\eqref{IO_alg} is essentially determined by whether the spectral radius of its iteration matrix is less than one, the following theorem gives a necessary and sufficient condition for its convergence.

\begin{theorem}\label{pu_convge_condi}
Assume that matrix equation \eqref{sto_lya_eq} has a unique solution $X^{*}$. Let $\alpha \in \mathbb{R}$ and $l\geq 1$. Then, for any initial matrix $X_0$, the sequence ${X_k}$ generated by the proposed IO iterative algorithm \eqref{IO_alg} converges to $X^{*}$ if and only if the parameters $\alpha$ and $l$ are selected such that 
	\begin{equation}\label{eq:necessary_sufficient_condition}
		\rho(R_{\alpha,l})<1,
	\end{equation}
	where
	\begin{equation}
		R_{\alpha,l}
		=
		(\alpha\Phi)^l
		+
		(1-\alpha)
		\sum_{s=0}^{l-1}(\alpha\Phi)^s\Phi .
		\label{eq:iteration_matrix_R_alpha_l}
	\end{equation}
	Equivalently, if $\mu_i$, $i \in \mathbb{I}[1,\,n^2]$ are the eigenvalues of $\Phi$, then the IO iteration algorithm \eqref{IO_alg} converges if and only if
	\begin{equation}
		\left|
		(\alpha\mu_i)^l
		+
		(1-\alpha)\mu_i
		\sum_{s=0}^{l-1}(\alpha\mu_i)^s
		\right|
		<1, \quad i \in \mathbb{I}[1,\,n^2].
		\label{eq:eigenvalue_condition}
	\end{equation}
\end{theorem}

\begin{proof}
	Let
	$
	x_k=\operatorname{vec}(X_k)
	$
	and
	$
	q=\operatorname{vec}(Q).
	$
	Vectorizing both sides of \eqref{IO_alg}, it follows that
	\[
	x_{k,j+1}
	=
	\alpha\Phi x_{k,j}
	+
	(1-\alpha)\Phi x_k
	+
	q,
	\quad j\in \mathbb{I}[0,\,l-1].
	\]
	Then, the proposed IO iteration \eqref{IO_alg} admits the form
	\begin{equation}
		x_{k+1}=R_{\alpha,l}x_k+S_{\alpha,l}q,
		\label{eq:linear_iteration}
	\end{equation}
	where $R_{\alpha,l}$ is defined in \eqref{eq:iteration_matrix_R_alpha_l} and
	$
	S_{\alpha,l}
	=
	\sum_{s=0}^{l-1}(\alpha\Phi)^s .
	$
	
	Let
	$
	x^{*}=\operatorname{vec}(X^{*}).
	$
    It is clear that
	\[
	(I-R_{\alpha,l})x^{*}
	=
	S_{\alpha,l}(I-\Phi)x^{*}
	=
	S_{\alpha,l}q.
	\]
	which implies
	\[
	x^{*}=R_{\alpha,l}x^{*}+S_{\alpha,l}q.
	\]
	
	Let $e_k=x_k-x^{*}$. Subtracting the above equality from \eqref{eq:linear_iteration}, it follows that
	\begin{equation}\label{eq:error_iteration}
		e_{k+1}=R_{\alpha,l}e_k.
	\end{equation}
	This indicates that the IO iteration algorithm \eqref{IO_alg} converges to the unique solution $X^{*}$ if and only if  the condition \eqref{eq:necessary_sufficient_condition} holds.
	
	Let $\mu_i$, $i \in \mathbb{I}[1,\,n^2]$ are the eigenvalues of $\Phi$. Clearly, $g_{\alpha,l}(\mu_i)$ are the eigenvalues of $R_{\alpha,l}$, 
	where
	\begin{equation}\label{mat_poly}
		g_{\alpha,l}(z)
		=
		(\alpha z)^l
		+
		(1-\alpha)z
		\sum_{s=0}^{l-1}(\alpha z)^s.
	\end{equation}
	Hence, the condition \eqref{eq:necessary_sufficient_condition} is equivalent to the eigenvalue condition \eqref{eq:eigenvalue_condition}. This completes the proof.
\end{proof}

\begin{remark}
It follows from \textbf{\Cref{pu_convge_condi}} that the restriction $\alpha\in(0,1]$ imposed in \textbf{\Cref{th_mo_bo}} is not necessary. Moreover, the zero initial condition can be relaxed to an arbitrary initial condition.
\end{remark}

Although \textbf{\Cref{pu_convge_condi}} gives a necessary and sufficient condition for the convergence of Algorithm \eqref{IO_alg}, this condition is expressed in terms of the spectral radius of the iteration matrix. To obtain a more easily verifiable convergence criterion, the case where $\Phi$ is nonnegative is further considered. In this case, the convergence of Algorithm \eqref{IO_alg} can be determined by checking whether the spectral radius of $\Phi$ is less than one.

\begin{theorem}\label{conv_th_phi}
Suppose that matrix equation \eqref{sto_lya_eq} has a unique solution. Assume that $\Phi\geq 0$, $0<\alpha\leq1$, and $l\geq 1$. Then, for any initial condition, the sequence generated by the IO iteration \eqref{IO_alg} converges to the solution of matrix equation \eqref{sto_lya_eq} if and only if
\[
\rho(\Phi)<1.
\]

\end{theorem}

	\begin{proof}
	Since the sufficiency is straightforward, only the necessity needs to be proved.
	
	 Suppose that the sequence generated by the IO iteration \eqref{IO_alg} converges and that $\Phi\geq 0$. By \textbf{\Cref{Perro_Fro}}, $r=\rho(\Phi)$ is a nonnegative real eigenvalue of $\Phi$. This shows that $g_{\alpha,l}(r)$ is an eigenvalue of $R_{\alpha,l}$, where $g_{\alpha,l}$ is defined in \eqref{mat_poly}. If $r\geq 1$, then
	\[
	g_{\alpha,l}(r)
	=
	(1-\alpha)r
	+
	\alpha(1-\alpha)r^2
	+
	\cdots
	+
	\alpha^{l-1}r^l
	\geq
	1.
	\]
	This contradicts $\rho(R_{\alpha,l})<1$. Hence, $r=\rho(\Phi)<1$. This completes the proof.
\end{proof}

In the general case, \textbf{\Cref{pu_convge_condi}} determines the convergence of Algorithm \eqref{IO_alg} through an implicit condition involving the parameter $\alpha$. For the case $l=2$, an explicit admissible interval for the parameter $\alpha$ can be derived to guarantee the convergence of Algorithm \eqref{IO_alg}.

\begin{theorem}\label{thm:l2_real_eigenvalue}\label{alpha_interval_real}
	Assume that matrix equation \eqref{sto_lya_eq} has a unique solution. Denote the spectrum of $\Phi$ by
	\[
	\sigma(\Phi)=\{\mu_i\in\mathbb{R}: \mu_i\neq 1,\ i\in \mathbb{I}[1,\,n^2]\}.
	\]
	Then, for any initial condition, the sequence generated by the IO iteration \eqref{IO_alg} with $l=2$ converges to the solution of matrix equation \eqref{sto_lya_eq} if and only if
	
	\[
	\alpha\in \bigcap_{\mu_i\neq 0}(\alpha_i^-,\alpha_i^+),
	\]
	where
	\[
	\alpha_i^-=
	\begin{cases}
		-\dfrac{1}{\mu_i}, & 0<\mu_i<1,\\[2mm]
		-\dfrac{1+\mu_i}{\mu_i(\mu_i-1)}, & \mu_i<0 \quad or \quad \mu_i >1,
	\end{cases}
	\]
	and
	\[
	\alpha_i^+=
	\begin{cases}
		\dfrac{1+\mu_i}{\mu_i(1-\mu_i)}, & 0<\mu_i<1,\\[2mm]
		-\dfrac{1}{\mu_i}, & \mu_i<0 \quad or \quad \mu_i >1.
	\end{cases}
	\]
For the zero eigenvalue $\mu_i=0$, no restriction is imposed on $\alpha$.

\end{theorem}

\begin{proof}
	For $l=2$, the eigenvalues of the iteration matrix of Algorithm \eqref{IO_alg} are given by
	$
	g_{\alpha,2}(\mu_i)
	=
	(1-\alpha)\mu_i+\alpha\mu_i^2,
    $
	where $g_{\alpha,2}$ is defined in \eqref{mat_poly}. Hence, by \textbf{\Cref{pu_convge_condi}}, the sequence generated by Algorithm \eqref{IO_alg} converges to the solution of matrix equation \eqref{sto_lya_eq} for any initial condition if and only if the parameter $\alpha$ is selected such that
	\[
	\left|\mu_i+\alpha\mu_i(\mu_i-1)\right|<1, \quad i \in \mathbb{I}[1,\,n^2].
	\]
	
    First, consider the case $0<\mu_i<1$. It follows directly that the admissible interval is
	\[
	\alpha\in
	\left(
	-\frac{1}{\mu_i},
	\frac{1+\mu_i}{\mu_i(1-\mu_i)}
	\right).
	\]
    Next, consider the case $\mu_i<0$ or $\mu_i>1$. It follows directly that the admissible interval is
	\[
	\alpha\in
	\left(
	\frac{1+\mu_i}{\mu_i(1-\mu_i)},
	-\frac{1}{\mu_i}
	\right).
	\]
	Combining the two cases, the convergence condition associated with each nonzero real eigenvalue $\mu_i$ is
	\[
	\alpha\in(\alpha_i^-,\alpha_i^+).
	\]
	For $\mu_i=0$, it holds that $|g_{\alpha,2}(0)|=0<1$. This shows that the zero eigenvalue imposes no restriction on $\alpha$. This completes the proof.
\end{proof}

\textbf{\Cref{alpha_interval_real}} discusses the case where all eigenvalues of $\Phi$ are real. When $\Phi$ has complex eigenvalues, an admissible interval for the parameter $\alpha$ can also be obtained, as stated in the following theorem.

\begin{theorem}\label{thm:l2_alpha_necessary_sufficient}
	Assume that matrix equation \eqref{sto_lya_eq} has a unique solution. Denote the spectrum of $\Phi$ by
	\[
	\sigma(\Phi)=\{\mu_i=a_i+\mathrm{i}b_i,\ i \in \mathbb{I}[1,\,n^2]\}.
	\]
	Suppose further that each nonzero eigenvalue $\mu_i$ satisfies $|\mu_i|<1$. Then, for any initial condition, the sequence generated by the IO iteration \eqref{IO_alg} with $l=2$ converges to the solution of matrix equation \eqref{sto_lya_eq} if and only if the parameter $\alpha$ is selected such that
	\[
	\alpha\in \bigcap_{\mu_i\neq 0}(\tilde{\alpha}_i^-,\tilde{\alpha}_i^+),
	\]
	where
	\[
	\tilde{\alpha}_i^\pm=
	\frac{-\mathcal B_i\pm\sqrt{\Delta_i}}{2\mathcal A_i},
	\]
	
	\[
	\mathcal A_i =(a_i^2+b_i^2)\big((1-a_i)^2+b_i^2\big),
	\]
	\[
	\mathcal B_i=2(a_i^2+b_i^2)(a_i-1),
	\]
	\[
	\mathcal C_i=a_i^2+b_i^2-1,
	\]
	and
	\[
	\Delta_i=\mathcal B_i^2-4\mathcal A_i \mathcal C_i.
	\]
	For the zero eigenvalue $\mu_i=0$, no restriction on $\alpha$ is imposed.
\end{theorem}

\begin{proof}
For $l=2$, the eigenvalues of the iteration matrix of the algorithm \eqref{IO_alg} are given by
$
g_{\alpha,2}(\mu_i)
=
(1-\alpha)\mu_i+\alpha\mu_i^2,
$
where $g_{\alpha,2}$ is defined in \eqref{mat_poly}. Therefore, by \textbf{\Cref{pu_convge_condi}}, the convergence condition associated with $\mu_i$ becomes
	\[
	(a_i^2+b_i^2)
	\left(
	[1-\alpha(1-a_i)]^2+\alpha^2b_i^2
	\right)<1, \quad i \in \mathbb{I}[1,\,n^2],
	\]
	Equivalently,
	\[
	\mathcal A_i\alpha^2+ \mathcal B_i\alpha+ \mathcal C_i<0.
	\]
	The result follows by analyzing the above quadratic inequality.
	
	For $\mu_i=0$, it holds that $|g_{\alpha,2}(0)|=0<1$. Therefore, the zero eigenvalue imposes no restriction on $\alpha$. This completes the proof.
\end{proof}

\begin{remark}
    When $l>2$, the eigenvalues of the iteration matrix are given by
    $
    (\alpha\mu_i)^l
    +
    (1-\alpha)\mu_i
    \sum_{s=0}^{l-1}(\alpha\mu_i)^s .
    $
    Consequently, determining the convergence condition reduces to solving a higher-degree inequality with respect to $\alpha$. In future work, methods for solving higher-degree inequalities, such as those based on Viète's formulas, will be employed to determine the admissible interval for the parameter $\alpha$.
    
\end{remark}

\section{Selection of the optimal parameters}\label{opt}

Discussed in this section are the choices of the optimal parameter $\alpha$ and the inner iteration number $l$.

For a fixed $l$, by \textbf{\Cref{pu_convge_condi}}, the admissible set of the tuning parameter $\alpha$ is defined as
\[
\Omega_l
=
\{\alpha\in\mathbb R:\rho(R_{\alpha,l})<1\}.
\label{eq:admissible_set_general}
\]

 In general, a smaller spectral radius leads to a faster convergence rate. Therefore, for a fixed number of inner iterations $l$, the optimal tuning parameter can be defined as:
\[
	\alpha_{\mathrm{opt}} =
\arg\min_{\alpha\in\Omega_l}
\max_{\mu_i\in\sigma(\Phi)}
\left|
(\alpha\mu_i)^l
+
(1-\alpha)\mu_i
\sum_{s=0}^{l-1}(\alpha\mu_i)^s
\right|.
\]

For a fixed admissible parameter $\alpha$, the number of inner iterations $l$ also affects the convergence speed. A larger $l$ may reduce $\rho(R_{\alpha,l})$, but it also increases the computational cost of each outer iteration. Conversely, if $l$ is too small, the inner iteration \eqref{sto_inner_eq} may not provide a sufficiently accurate approximation to the solution $X_{k+1}$ of the outer iteration equation \eqref{eq:outer_matrix}. Therefore, the optimal choice of $l$ should balance the reduction in the spectral radius and the additional cost of inner iterations. Accordingly, a reasonable criterion is defined as follows:
\[
l_{\mathrm{opt}}
=
\arg\max_{l\in\mathbb{N}}
\frac{-\ln \rho(R_{\alpha,l})}{l}.
\]

In future work, explicit expressions for the optimal parameter $\alpha_{\mathrm{opt}}$ and the optimal number of inner iterations $l_{\mathrm{opt}}$ will be investigated in the general case. This paper focuses mainly on the selection of the optimal parameter $\alpha_{\mathrm{opt}}$ when the number of inner iterations is fixed at $l=2$.

\begin{proposition}\label{pro_opt}
Assume that matrix equation \eqref{sto_lya_eq} has a unique solution. For Algorithm \eqref{IO_alg} with $l=2$, let
\[
\sigma(\Phi)=\{\mu_i:\ i\in \mathbb{I}[1,\,n^2]\}.
\]
If $\mu_i\in\mathbb{R}$ and $\mu_i\neq 1$ for all $i\in \mathbb{I}[1,\,n^2]$, then the optimal parameter $\alpha_{\mathrm{opt}}$ can be found from the following finite candidate set,

\[
		\label{Delta1}
		\begin{aligned}
			\Delta_1
			=&
			\left\{
			-\frac{\widetilde{\mathcal B}_i}{2\widetilde{\mathcal A}_i}
			:
			-\frac{\widetilde{\mathcal B}_i}{2\widetilde{\mathcal A}_i}
			\in I
			\right\}      \\
			&\cup
			\left\{
			\alpha:
			\begin{aligned}
				&
				(\widetilde{\mathcal A}_i-\widetilde{\mathcal A}_j)\alpha^2
				+
				(\widetilde{\mathcal B}_i-\widetilde{\mathcal B}_j)\alpha  \\
				&\quad
				+
				(\widetilde{\mathcal D}_i-\widetilde{\mathcal D}_j)=0,
				\quad
				\alpha\in I,
				\quad i<j
			\end{aligned}
			\right\},
		\end{aligned}
\]
	where
    $
	\widetilde{\mathcal A}_i=\mu_i^2(1-\mu_i)^2,
	\widetilde{\mathcal B}_i=2\mu_i^2(\mu_i-1),
	\widetilde{\mathcal D}_i=\mu_i^2,
    $
    and
    $I = \bigcap_{\mu_i\neq 0}
    (\alpha_i^-,\alpha_i^+)$.
	
	If \(\mu_i=a_i+\mathrm{i}b_i\) and \(|\mu_i|<1\) for all nonzero eigenvalues,
	then the optimal parameter \(\alpha_{\mathrm{opt}}\) can be found in the following
	finite candidate set,
\[
\label{Delta2}
\begin{aligned}
	\Delta_2
	=&
	\left\{
	-\frac{\mathcal B_i}{2\mathcal A_i}
	:
	-\frac{\mathcal B_i}{2\mathcal A_i}
	\in J
	\right\}      \\
	&\cup
	\left\{
	\alpha:
	\begin{aligned}
		&
		(\mathcal A_i-\mathcal A_j)\alpha^2
		+
		(\mathcal B_i-\mathcal B_j)\alpha  \\
		&\quad
		+
		(\mathcal D_i-\mathcal D_j)=0,
		\quad
		\alpha\in J,
		\quad i<j
	\end{aligned}
	\right\},
\end{aligned}
\]

	where $\mathcal A_i, \mathcal B_i$ are defined in \textbf{\Cref{thm:l2_alpha_necessary_sufficient}},
	$
	\mathcal D_i=a_i^2+b_i^2,
	$
	and $J = \bigcap_{\mu_i\neq 0}
	(\tilde{\alpha}_i^-,\tilde{\alpha}_i^+)$.
\end{proposition}

\begin{proof}
	It suffices to prove the case of real eigenvalues, since the proof for the case of complex eigenvalues is similar.
For $l=2$, the eigenvalues of the iteration matrix of Algorithm \eqref{IO_alg} are given by
	\[
	g_{\alpha,2}(\mu_i)
	=
	\mu_i+\alpha\mu_i(\mu_i-1),
	\]
	where $g_{\alpha,2}$ is defined in \eqref{mat_poly}. Taking the square of the modulus gives
	\[
	|g_{\alpha,2}(\mu_i)|^2
	=
	\widetilde{\mathcal A}_i\alpha^2+\widetilde{\mathcal B}_i\alpha+\widetilde{\mathcal D}_i. 
	\]
	Define the auxiliary functions
	$
	f_i(\alpha)=|g_{\alpha,2}(\mu_i)|^2.
	$
	Then, the optimal parameter $\alpha_{\mathrm{opt}}$ can be obtained by solving the following minimax problem,
	\[
	\min_{\alpha\in \bigcap_{\mu_i\neq 0}(\alpha_i^-,\alpha_i^+)}
	\max_{i\in \mathbb{I}[1,n^2]} f_i(\alpha).
	\]
	
	By analyzing the graphs of the quadratic functions $f_i(\alpha)$, the optimal parameter can occur only in the following two cases. In the first case, $\alpha_{\mathrm{opt}}$ is located on the axis of symmetry of some quadratic function $f_i(\alpha)$. In the second case, $\alpha_{\mathrm{opt}}$ is located at an intersection point of two quadratic functions $f_i(\alpha)$ and $f_j(\alpha)$, where $i<j$. Therefore, in the case of real eigenvalues, the optimal parameter can be searched for within the finite candidate set $\Delta_1$. This completes the proof.
\end{proof}

By \textbf{Proposition \ref{pro_opt}}, it is sufficient to search for the optimal parameter within the candidate set $\Delta_1$ or $\Delta_2$. Therefore, a heap sort algorithm can be used to find the optimal parameter from the candidate set. The specific procedure is given as follows.

\begin{algorithm}[!t]
	\caption{Heap Sort}
	\label{alg:heap_sort}
	\small
	\setlength{\baselineskip}{11pt}
	\begin{algorithmic}[1]
		\State Compute the eigenvalues $\mu_i$, $i\in I[1,n^2]$, of the matrix $\Phi$.
		\If{$\mu_i\in\mathbb{R}$ for all $i\in I[1,n^2]$}
		\State Compute all candidates in $\Delta_1$ and store them in $\Gamma$.
		\ElsIf{$|\mu_i|<1$ for all nonzero eigenvalues $\mu_i$}
		\State Compute all candidates in $\Delta_2$ and store them in $\Gamma$.
		\Else
		\State The candidate set $\Delta_2$ is not applicable.
		\State Stop.
		\EndIf
		\State Compute the corresponding spectral radius by using the numbers in array $\Gamma$ , and store them sequen
		tially in array $P(s)$.
		\State Store the ordered pair $(\Gamma(s),P(s))$ in $S$.
		\State Construct a min-heap from $S$ according to $P(s)$.
		\State Let $N=\operatorname{length}(\Gamma)$.
		\For{$s=\lfloor N/2\rfloor:-1:1$}
		\State Call $\operatorname{HEAPIFY}(S,s,N)$.
		\EndFor
		\State Let $(\Gamma(1),P(1))$ be the root of the min-heap.
		\State Set $\alpha_{\mathrm{opt}}=\Gamma(1)$ and $\rho_{\min}=P(1)$.
		\State Output $\alpha_{\mathrm{opt}}$ and $\rho_{\min}$.
		
		\Procedure{HEAPIFY}{$S,s,N$}
		\State Set $smallest=s$.
		\State Set $left=2s$ and $right=2s+1$.
		\If{$left\leq N$ and $P(left)<P(smallest)$}
		\State $smallest=left$.
		\EndIf
		\If{$right\leq N$ and $P(right)<P(smallest)$}
		\State $smallest=right$.
		\EndIf
		\If{$smallest\neq s$}
		\State Exchange $S(s)$ and $S(smallest)$.
		\State Call $\operatorname{HEAPIFY}(S,smallest,N)$.
		\EndIf
		\EndProcedure
	\end{algorithmic}
\end{algorithm}

\section{Numerical examples}\label{example}
In this section, several numerical examples are presented to show the
effectiveness of the proposed algorithm \eqref{IO_alg}. The residual is defined by
\[
\Delta(k) = \left\| A_0^{\mathrm{T}} X(k) A_0 + \sum_{i=1}^{m} \delta_i A_i^{\mathrm{T}} X(k) A_i + Q - X(k) \right\|_{\mathrm{F}}.
\]

\begin{example}\label{eq1}
	In this example, SLM equation \eqref{sto_lya_eq} with a single random disturbance term is considered, where $m=1$ and $\delta_1=1$. The coefficient matrices are given as follows.
	\[
	A_0 =
	\begin{bmatrix}
		0.1679 &  0.0827 &  0.1016 & -0.1377 & -0.0529 \\
		0.0827 &  0.2397 & -0.0389 &  0.0344 & -0.0974 \\
		0.1016 & -0.0389 &  0.1305 &  0.0585 &  0.1092 \\
		-0.1377 &  0.0344 &  0.0585 &  0.2053 &  0.1558 \\
		-0.0529 & -0.0974 &  0.1092 &  0.1558 &  0.0961
	\end{bmatrix},
	\]
	\[
	A_1 =
	\begin{bmatrix}
		0.4525 &  0.0309 & -0.0605 &  0.2557 &  0.0443 \\
		0.0309 &  0.4097 &  0.2026 & -0.0115 &  0.1808 \\
		-0.0605 &  0.2026 &  0.4839 &  0.0218 & -0.0741 \\
		0.2557 & -0.0115 &  0.0218 &  0.4212 & -0.1589 \\
		0.0443 &  0.1808 & -0.0741 & -0.1589 &  0.4953
	\end{bmatrix}.
	\]
	and $Q$ is chosen as the identity matrix. 
	
In this example, direct computation shows that all eigenvalues of $\Phi$ are real. Therefore, the eigenvalues of $\Phi$ satisfy the condition in \textbf{\Cref{alpha_interval_real}}. When $l=2$, Algorithm \eqref{IO_alg} converges to the unique solution of SLM equation \eqref{sto_lya_eq} for any initial condition if and only if
$
\alpha\in(-1.7790,\,5.8549).
$
Then, the heap sort algorithm is applied over this interval, yielding the optimal parameter $\alpha_{\mathrm{opt}}=1.8754$. 
\textbf{Fig.~\ref{fig:numerical_results} (a)} shows that the spectral radius $\rho(R_{\alpha,2})$ attains its minimum at $\alpha_{\mathrm{opt}}=1.8754$, which demonstrates the effectiveness of the heap sort algorithm for optimal parameter selection.
Next, the convergence performance of Algorithm \eqref{IO_alg} with different values of $\alpha$ is compared. 
The convergence curves with the zero initial condition, $l=2$, and $\alpha \in (-1.7790,\,5.8549)$ are shown in \textbf{Fig.~\ref{fig:numerical_results} (b)}. 
\textbf{Fig.~\ref{fig:numerical_results} (b)} shows that Algorithm \eqref{IO_alg} achieves the fastest convergence rate when $\alpha_{\mathrm{opt}}=1.8754$.

\begin{figure*}[!t]
	\centering
	
	\begin{minipage}[t]{0.32\textwidth}
		\centering
		\includegraphics[width=\textwidth]{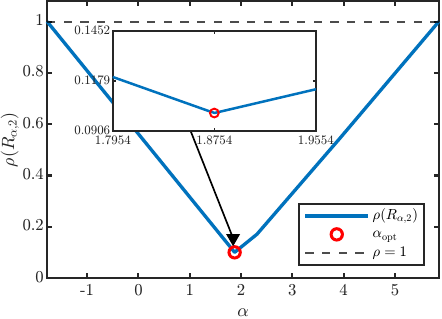}
		
		\vspace{1mm}
		\small (a)
	\end{minipage}
	\hfill
	\begin{minipage}[t]{0.32\textwidth}
		\centering
		\includegraphics[width=\textwidth]{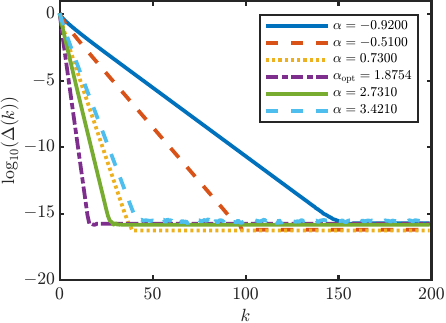}
		
		\vspace{1mm}
		\small (b)
	\end{minipage}
	\hfill
	\begin{minipage}[t]{0.32\textwidth}
		\centering
		\includegraphics[width=\textwidth]{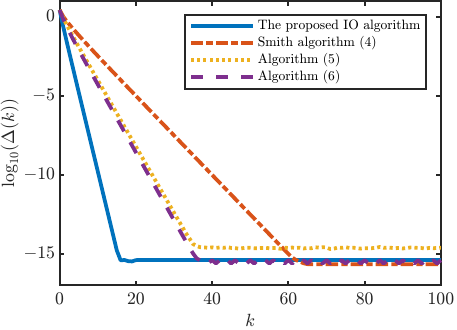}
		
		\vspace{1mm}
		\small (c)
	\end{minipage}
	
	\caption{
		Convergence behavior in Example~\ref{eq1}: 
		(a) Spectral radius of the iteration matrix $R_{\alpha,2}$ with respect to $\alpha$;
		(b) Convergence curves of Algorithm~\eqref{IO_alg} with different parameters $\alpha$;
		(c) Comparison of different algorithms.
	}
	\label{fig:numerical_results}
\end{figure*}

	Moreover, the convergence performance of Algorithm \eqref{IO_alg} with $\alpha_{\mathrm{opt}}=1.8754$ and $l=2$ is compared with that of the existing algorithms \eqref{simith}, \eqref{zhan}, and \eqref{chen}. \textbf{Fig.~\ref{fig:numerical_results} (c)} presents the convergence curves of these algorithms. As shown in \textbf{Fig.~\ref{fig:numerical_results} (c)}, Algorithm \eqref{IO_alg} requires fewer iterations than the existing algorithms.

	Furthermore, the convergence performance of Algorithm \eqref{IO_alg} is further compared with that of the existing algorithms in terms of the iteration steps (IT), average iteration time (Time), and final iteration error $\Delta(k)$ under the stopping criterion $\Delta(k)<10^{-12}$. The results are listed in \textbf{Table \ref{tab:algorithm_comparison}}. \textbf{Table \ref{tab:algorithm_comparison}} shows that Algorithm \eqref{chen} requires less average iteration time than Algorithm \eqref{IO_alg}. However, it requires more iterations and yields a larger final iteration error. In particular, its final iteration error is much larger than that of Algorithm \eqref{IO_alg}. Moreover, Algorithm \eqref{IO_alg} outperforms Algorithms \eqref{simith} and \eqref{zhan} in terms of all three metrics. 
	
	\begin{table}[!htbp]
		\centering
		\captionsetup{justification=centering,singlelinecheck=false}
		\caption{Comparison of different iterative algorithms with $\Delta(k)<10^{-12}$ in Example \ref{eq1}.}
		\label{tab:algorithm_comparison}
		\footnotesize
		\setlength{\tabcolsep}{4pt}
		\renewcommand{\arraystretch}{1.1}
		\begin{tabular}{l c c c}
			\toprule
			Algorithm & IT  & Time, s  & Error  \\
			\midrule
			The proposed algorithm \eqref{IO_alg} & 13 & 4.7295e-05 & 1.6532e-13 \\
			Smith algorithm \eqref{simith}        & 48 & 7.2587e-05 & 9.8147e-13 \\
			Implicit algorithm \eqref{zhan}       & 30 & 0.0024     & 4.0242e-13 \\
			Explicit algorithm \eqref{chen}       & 28 & 4.5542e-05 & 9.6399e-13 \\
			\bottomrule
		\end{tabular}
	\end{table}
	
   Finally, to investigate the influence of $l$ on Algorithm \eqref{IO_alg}, \textbf{Table~\ref{tab:different_l}} reports the convergence performance for different values of $l$ under the zero initial condition and a fixed $\alpha$, with the stopping criterion $\Delta(k)<10^{-12}$. \textbf{Table~\ref{tab:different_l}} shows that, as $l$ increases, the number of iteration steps decreases, whereas the computational cost increases. Therefore, a suitable choice of $l$ is necessary to achieve an appropriate balance between convergence efficiency and computational cost.

\begin{table}[!htbp]
	\centering
	\caption{Convergence performance of Algorithm~\eqref{IO_alg} for different $l$ with fixed $\alpha=0.8$ in Example \ref{eq1}.}
	\label{tab:different_l}
	\setlength{\tabcolsep}{6pt}
	\renewcommand{\arraystretch}{1.1}
	\begin{tabular*}{0.80\columnwidth}{@{\extracolsep{\fill}}ccc@{}}
		\toprule
		$l$ & IT & Time, s \\
		\midrule 
		3 & 22 & 8.1130e-05 \\
		4 & 20 & 8.3446e-05 \\
		5 & 19 & 8.9038e-05 \\
		6 & 18 & 9.6812e-05 \\
		7 & 18 & 1.0594e-04 \\
		\bottomrule
	\end{tabular*}
\end{table}

\end{example}

\begin{example}\label{eq2}
	
In this example, numerical experiments with increasing matrix dimensions are carried out to compare the convergence performance of Algorithm \eqref{IO_alg} and Algorithm \eqref{chen}. For each dimension $n$, the matrices $A_0$ and $A_1$ are generated by pole assignment. Specifically, the eigenvalues are prescribed in advance, and the corresponding matrices are constructed to ensure $\rho(\Phi)<1$.

\textbf{Table~\ref{tab:dimension_comparison}} reports the iteration steps (IT), average iteration time (Time), and final iteration error $\Delta(k)$ with $\Delta(k)<10^{-12}$, where the optimal $\alpha$ of Algorithm \eqref{IO_alg} is obtained by heap sort algorithm over the admissible interval from \textbf{\Cref{alpha_interval_real}}, showing fewer iterations and smaller final errors than Algorithm \eqref{chen}, with comparable Time.

\begin{table*}[!t]
	\centering
	
	\caption{Convergence comparison for different dimensions with $\Delta(k)<10^{-12}$ in Example~\ref{eq2}.}
	\label{tab:dimension_comparison}
	\small
	\renewcommand{\arraystretch}{1.05}
	\begin{tabular*}{\textwidth}{@{}c@{\extracolsep{\fill}}cccccccc@{}}
		\toprule
		\multirow{2}{*}{$n$}
		& \multicolumn{4}{c}{IO algorithm (12)}
		& \multicolumn{4}{c}{Algorithm (6)} \\
		\cmidrule(lr){2-5} \cmidrule(lr){6-9}
		& $\alpha_{\mathrm{opt}}$
		& IT
		& Time, s
		& Error
		& $\gamma_{\mathrm{opt}}$
		& IT
		& Time, s
		& Error \\
		\midrule
		50   & 1.5249 & 10 & 4.930e-04 & 1.328e-13 & 1.3599 & 19 & 4.6927e-04 & 2.0932e-13 \\
		100  & 1.5319 & 10 & 2.0217e-03 & 2.1918e-13 & 1.3601 & 19 & 2.0093e-03 & 3.268e-13 \\
		300  & 1.5528 & 10 & 1.4516e-02 & 7.1187e-13 & 1.3658 & 19 & 1.5075e-02 & 8.970e-13 \\
		500  & 1.5427 & 10 & 0.1378 & 6.8962e-13 & 1.3625 & 19 & 0.1329 & 9.248e-13 \\
		700  & 1.5356 & 11 & 0.3275 & 5.7163e-14 & 1.3674 & 20 & 0.3093 & 3.0813e-13 \\
		1000 & 1.5500 & 11 & 0.6417 & 5.7626e-14 & 1.3660 & 20 & 0.5863 & 3.2838e-13 \\
		\bottomrule
	\end{tabular*}
\end{table*}
\end{example}

\section{Conclusion}\label{conclu}
This paper proposes an inner-outer iterative algorithm for solving stochastic Lyapunov matrix equations. The proposed method introduces the number of inner iteration steps as an additional adjustable parameter, thereby improving the flexibility of parameter tuning. Convergence properties are established under both zero and arbitrary initial conditions. For the case $l=2$, an explicit admissible interval for the parameter $\alpha$ is derived, and heap sort algorithm are provided to determine the optimal parameter. Numerical examples verify the theoretical results and show that the proposed algorithm achieves better convergence performance than several existing iterative methods. Future work will focus on deriving explicit expressions for the optimal parameters in the general case.

\bibliographystyle{cas-model2-names}

\bibliography{cas-refs}

@article{digalakis2002ml,
	title={ML estimation of a stochastic linear system with the EM algorithm and its application to speech recognition},
	author={Digalakis, Vassilios and Rohlicek, Jan Robin and Ostendorf, Mari},
	journal={IEEE Transactions on speech and audio processing},
	volume={1},
	number={4},
	pages={431--442},
	year={2002},
	publisher={IEEE}
}

@article{gallant2025soft,
	title={Soft-Constrained Stochastic MPC of Markov Jump Linear Systems: Application to Real-Time Control With Deadline Overruns},
	author={Gallant, Melanie and Mark, Christoph and Pazzaglia, Paolo and von Keler, Johannes and Beermann, Laura and Schmidt, Kevin and Maggio, Martina},
	journal={IEEE Control Systems Letters},
	year={2025},
	publisher={IEEE}
}

@article{mclane1969asymptotic,
	title={Asymptotic stability of linear autonomous systems with state-dependent noise},
	author={McLane, P},
	journal={IEEE Transactions on Automatic Control},
	volume={14},
	number={6},
	pages={754--755},
	year={1969},
	publisher={IEEE}
}

@article{jodar1987explicit,
	title={Explicit solutions for a system of coupled Lyapunov differential matrix equations},
	author={Jodar, L and Mariton, M},
	journal={Proceedings of the Edinburgh Mathematical Society},
	volume={30},
	number={3},
	pages={427--434},
	year={1987},
	publisher={Cambridge University Press}
}

@article{ding2005gradient,
	title={Gradient based iterative algorithms for solving a class of matrix equations},
	author={Ding, Feng and Chen, Tongwen},
	journal={IEEE Transactions on Automatic Control},
	volume={50},
	number={8},
	pages={1216--1221},
	year={2005},
	publisher={IEEE}
}

@article{zhou2010gradient,
	title={Gradient-based maximal convergence rate iterative method for solving linear matrix equations},
	author={Zhou, Bin and Lam, James and Duan, Guang-Ren},
	journal={International journal of computer mathematics},
	volume={87},
	number={3},
	pages={515--527},
	year={2010},
	publisher={Taylor \& Francis}
}

@article{zhang2017implicit,
	title={Implicit iterative algorithms with a tuning parameter for discrete stochastic Lyapunov matrix equations},
	author={Zhang, Ying and Wu, Ai-Guo and Shao, Chun-Tao},
	journal={IET Control Theory \& Applications},
	volume={11},
	number={10},
	pages={1554--1560},
	year={2017},
	publisher={Wiley Online Library}
}

@article{wu2018iterative,
	title={An iterative algorithm for discrete periodic Lyapunov matrix equations},
	author={Wu, Ai-Guo and Zhang, Wen-Xue and Zhang, Ying},
	journal={Automatica},
	volume={87},
	pages={395--403},
	year={2018},
	publisher={Elsevier}
}

@article{tian2018new,
	title={New explicit iteration algorithms for solving coupled continuous Markovian jump Lyapunov matrix equations},
	author={Tian, Zhaolu and Fan, CM and Deng, Youjun and Wen, PH},
	journal={Journal of the Franklin Institute},
	volume={355},
	number={17},
	pages={8346--8372},
	year={2018},
	publisher={Elsevier}
}

@article{tian2020multi,
	title={A multi-step Smith-inner-outer iteration algorithm for solving coupled continuous Markovian jump Lyapunov matrix equations},
	author={Tian, Zhaolu and Wang, Junxin and Dong, Yinghui and Liu, Zhongyun},
	journal={Journal of the Franklin Institute},
	volume={357},
	number={6},
	pages={3656--3679},
	year={2020},
	publisher={Elsevier}
}

@article{he2025convergence,
	title={Convergence properties of Jacobi gradient-based iteration algorithm for the complex conjugate and transpose Sylvester matrix equations},
	author={He, Jiating and Chen, Xuesong},
	journal={Numerical Algorithms},
	year={2025},
	publisher={Springer},
	note={\url{https://doi.org/10.1007/s11075-025-02277-5}}
}

@article{chen2026explicit,
	title={Explicit iterative algorithm with optimal tuning parameter for stochastic Lyapunov matrix equation},
	author={Chen, Zebin and Sun, Hui-Jie and Zhang, Jinxiu and Chen, Xuesong},
	journal={Numerical Algorithms},
	volume={101},
	number={1},
	pages={679--702},
	year={2026},
	publisher={Springer}
}

@article{kubrusly1985mean,
	title={Mean square stability conditions for discrete stochastic bilinear systems},
	author={Kubrusly, Carlos and Costa, O},
	journal={IEEE Transactions on Automatic Control},
	volume={30},
	number={11},
	pages={1082--1087},
	year={1985},
	publisher={IEEE}
}

@article{bibby1974axiomatisations,
	title={Axiomatisations of the average and a further generalisation of monotonic sequences},
	author={Bibby, John},
	journal={Glasgow Mathematical Journal},
	volume={15},
	number={1},
	pages={63--65},
	year={1974},
	publisher={Cambridge University Press}
}

@book{horn2012matrix,
	title={Matrix analysis},
	author={Horn, Roger A and Johnson, Charles R},
	year={2012},
	publisher={Cambridge university press}
}

\end{document}